\documentclass[11pt,a4paper,reqno]{amsart}

\usepackage{epsfig}
\usepackage{amsfonts}
\usepackage{amsmath}
\usepackage{amssymb}
\usepackage{amsthm}
\usepackage{color}
\usepackage{enumerate}
\usepackage[T1]{fontenc}
\usepackage[latin1]{inputenc}
\usepackage{ae,aecompl}
\usepackage{graphicx}
\usepackage{nicefrac}
\usepackage[includehead,includefoot,margin=25mm]{geometry}
\usepackage{faktor}
\usepackage{dsfont}
\usepackage[all]{xy}
\usepackage{mathrsfs}
\usepackage{textcomp}
\usepackage{cite}

\newtheorem{theorem}{Theorem}[section]
\newtheorem*{theorem*}{Theorem}
\newtheorem{corollary}[theorem]{Corollary}
\newtheorem{lemma}[theorem]{Lemma}
\newtheorem{proposition}[theorem]{Proposition}

{\theoremstyle{remark}
  \newtheorem{remark}[theorem]{Remark}}
{\theoremstyle{definition}
  \newtheorem{definition}[theorem]{Definition}

}

\def\N{\mathbb{N}}

\def\Z{\mathbb{Z}}
\def\P{\mathbb{P}}
\def\C{\mathbb{C}}

\def\II{I_\mathrm{hol}}

\newcommand{\CC}[0]{\ensuremath{\mathbb{C}}}
\newcommand{\ZZ}[0]{\ensuremath{\mathbb{Z}}}
\newcommand{\GA}[0]{\ensuremath{\mathbb{G}_{\mathrm{a}}}}

\newcommand{\RR}[0]{\ensuremath{\mathcal{R}}}

\newcommand{\QQ}[0]{\ensuremath{\mathbb{Q}}}
\newcommand{\TT}[0]{\ensuremath{\mathrm{T}}}
\newcommand{\KK}[0]{\ensuremath{\mathbb{C}}}
\newcommand{\OO}[0]{\ensuremath{\mathcal{O}}}

\newcommand{\Span}[0]{\ensuremath{\operatorname{Span}}}

\newcommand{\spec}[0]{\ensuremath{\operatorname{Spec}}}

\newcommand{\Aut}[0]{\ensuremath{\operatorname{Aut}}}

\newcommand{\spann}[0]{\ensuremath{\operatorname{span}}}

\newcommand{\rank}[0]{\ensuremath{\operatorname{rank}}}

\newcommand{\relint}[0]{\ensuremath{\operatorname{rel.int}}}

\newcommand{\VFalg}[0]{\ensuremath{\operatorname{VF}_{\mathrm{alg}}}}
\newcommand{\Liealg}[0]{\ensuremath{\operatorname{Lie}_{\mathrm{alg}}}}

\newcommand{\Xreg}[0]{\ensuremath{X^{\mathrm{reg}}}}
\newcommand{\Xsing}[0]{\ensuremath{X^{\mathrm{sing}}}}

\newcommand{\VFhol}[0]{\ensuremath{\operatorname{VF}_{\mathrm{hol}}}}
\newcommand{\Liehol}[0]{\ensuremath{\operatorname{Lie}_{\mathrm{hol}}}}

\begin{document}

\title{The algebraic density property for affine toric varieties}

\author{Frank Kutzschebauch}
\address{Mathematisches Institut, Universit\"at Bern, Sidlerstrasse 5,
  CH-3012 Bern, Switzerland.}%
\email{frank.kutzschebauch@math.unibe.ch}

\author{Matthias Leuenberger}
\address{Mathematisches Institut, Universit\"at Bern, Sidlerstrasse 5,
  CH-3012 Bern, Switzerland.}%
\email{matthias.leuenberger@math.unibe.ch}

\author{Alvaro Liendo} %
\address{Instituto de Matem\'atica y F\'\i sica, Universidad de Talca,
  Casilla 721, Talca, Chile.}%
\email{aliendo@inst-mat.utalca.cl}

\date{\today}

\thanks{
  \mbox{\hspace{11pt}}{\it Key words}: Density property, affine toric
  varieties, locally nilpotent derivations, holomorphic automorphisms,
  Lie algebras.\\
  \mbox{\hspace{11pt}}The first and second authors were partially
  supported by Schweizerischer Nationalfond Grant 200021-140235/1 and
  the third author was supported by Fondecyt project 11121151}

\begin{abstract}
  In this paper we generalize the algebraic density property to not
  necessarily smooth affine varieties relative to some closed
  subvariety containing the singular locus. This property implies the
  remarkable approximation results for holomorphic automorphisms of
  the Anders\'en-Lempert theory. We show that an affine toric variety
  $X$ satisfies this algebraic density property relative to a closed
  $\TT$-invariant subvariety $Y$ if and only if $X\setminus Y \neq
  \TT$. For toric surfaces we are able to classify those which posses
  a strong version of the algebraic density property (relative to the
  singular locus). The main ingredient in this classification is our
  proof of an equivariant version of Brunella's famous classification
  of complete algebraic vector fields in the affine plane.
\end{abstract}

\maketitle

\section{Introduction}

A remarkable property of the Euclidean space of dimension at least
two, that to a great extent compensates for the lack of partition of
unity for holomorphic automorphisms, was discovered by Anders\'en and
Lempert in early 1990's \cite{An90,AnLe92}, see also the work by
Forstneri\v{c} and Rosay \cite{FoRo93}. Since then, the theory of
Stein manifolds with very large holomorphic automorphism group is
called Anders\'en-Lempert theory.

The property was formalized by Varolin who named it the density
property (DP). A Stein manifold $X$ has the DP if the Lie algebra
generated by complete holomorphic vector fields is dense (in the
compact-open topology) in the space of all holomorphic vector fields
on $X$. Recall that a vector field is called complete if its flow
exits for all complex time and all initial conditions.

The DP allows to construct (global) automorphisms of $X$ with
prescribed local properties. More precisely, any local phase flow on a
Runge domain in $X$ can be approximated by (global)
automorphisms. This has remarkable applications for geometric
questions in complex analysis, we refer the reader to survey articles
\cite{Ro99,KK11,Ku14} and the recent book \cite{Forst11}. For smooth
affine algebraic varieties, the algebraic density property (ADP) was
also introduced by Varolin. The ADP implies the DP, therefore it is
commonly used as a tool to prove the DP.

In this paper we generalize the ADP to not necessarily smooth affine
varieties relative to some closed subvariety containing the singular
locus as follows: Let $X$ be an affine algebraic variety and let
$\Xsing$ be the singular locus. We also let $Y\subseteq X$ be an
algebraic subvariety of $X$ containing $\Xsing$ and let
$I=I(Y)\subseteq \KK[X]$ be the ideal of $Y$.  Let $\VFalg(X,Y)$ be
the $\CC[X]$-module of vector fields vanishing in $Y$, i.e.,
$\VFalg(X,Y)=\{\partial \mid \partial(\CC[X])\subseteq I\}$. Let
$\Liealg(X,Y)$ be the Lie algebra generated by all the complete vector
fields in $\VFalg(X,Y)$.

\begin{definition} \label{ADP} %
  We say that $X$ has the strong ADP relative to $Y$ if $\VFalg(X,Y) =
  \Liealg(X,Y)$. Furthermore, we say that $X$ has the ADP relative to
  $Y$ if there exists $\ell\geq 0$ such that $I^\ell\VFalg(X,Y)
  \subseteq \Liealg(X,Y)$. With this definition, the ADP relative to
  $Y$ with $\ell=0$ is just the strong ADP relative to $Y$. If we let
  $Y=\Xsing$ we simply say that $X$ has the strong ADP or the ADP,
  respectively.
\end{definition}

Except for the fact that we consider not necessarily smooth varieties,
the strong ADP is a version of Varolin's Definition~3.1 in
\cite{varolin01} of DP for the Lie subalgebra of vector fields
vanishing on $Y$. Whereas for $\ell>0$ our property is slightly weaker
than Varolin's definition since we generate the Lie subalgebra of
vector fields vanishing on $Y$ of order at least $\ell$ using complete
vector fields vanishing on $Y$ of possibly lower order than
$\ell$. Still this version of the ADP has the same remarkable
consequences as in Varolin version of ADP for the group of holomorphic
automorphisms of $X$ fixing $Y$ pointwise (see
Theorem~\ref{AL-Theorem}).

In this paper we investigate the ADP for toric varieties. Our first
main result is the following theorem (see Theorem~\ref{finalthm}).

\begin{theorem*} 
  Let $X$ be an affine toric variety of dimension at least two and let
  $Y$ be a $\TT$-invariant closed subvariety of $X$ containing
  $\Xsing$. Then $X$ has the ADP relative to $Y$ if and only if
  $X\setminus Y\neq \TT$.
\end{theorem*}

Recall that every smooth affine toric variety is isomorphic
$\CC^k\times (\CC^*)^{n-k}$. A special case of our theorem where
$X=\CC^n$ and $Y$ is the union of up to $n-1$ coordinate hyperplanes
has been already proven by Varolin \cite{varolin01}.

It is well known that every affine toric surface different from $\C^*\times\C$ or $\C^*\times\C^*$ is obtained as a
quotient of $\C^2$ by the action of a cyclic group. Let $d>e$ be
relatively prime positive integers. We denote by $V_{d,e}$ the toric
surface obtained as the quotient of $\C^2$ by the $\Z_d$-action
$\zeta\cdot(u,v) = (\zeta u, \zeta^e v)$, where $\zeta$ is a primitive
$d$-th root of unity. The following theorem is our second main result
(see Corollary~\ref{Z-ADP}).

\begin{theorem*}
  $V_{d,e}$ has the strong ADP if and only if $e$ divides $d+1$ and
  $e^2 \neq d + 1$.
\end{theorem*}

Furthermore, for every affine toric surface our methods allow to
determine the values of $\ell$ from Definition~\ref{ADP} for which
$I^\ell\VFalg(X,\Xsing) \subseteq \Liealg(X,\Xsing)$. The main
ingredient in the proof of this theorem is an equivariant version of
Brunella's famous classification of complete algebraic vector fields
in the affine plane (see \cite{Br3}) or, equivalently, classification
of complete algebraic vector fields on affine toric surfaces (see
Theorem~\ref{thmlist}). This result might be of independent interest.

\section{Vector fields and the algebraic density property}

In this section we prove a general method for establishing the ADP
that we later will use to show the ADP for toric varieties.

\begin{definition}
  Let $X$ be an affine algebraic variety and $Y$ be a subvariety
  containing $\Xsing$.
  \begin{enumerate}[$(i)$]

  \item Let $\Aut(X,Y)$ be the subgroup of automorphism of $X$
    stabilizing $Y$. We say that $X$ is homogeneous with respect to
    $Y$ if $\Aut(X,Y)$ acts transitively on $X\setminus Y$.

  \item We also let $x_0\in \Xreg$. A finite subset $M$ of the tangent
    space $T_{x_0}X$ is called a generating set if the image of $M$
    under the action of the isotropy group of $x_0$ in $\Aut(X,Y)$
    generate the whole tangent space $T_{x_0}X$.
  \end{enumerate}
\end{definition}

The following is our main tool to establish the ADP for toric
varieties. It is a generalization of \cite[Theorem~1]{KaKu08}.

\begin{theorem}\label{thm}
  Let $X$ be an algebraic variety homogeneous with respect to some
  subvariety $Y\supseteq \Xsing$. Let also $L$ be a finitely generated
  submodule of the $\KK[X]$-module $\VFalg(X,Y)$ of vector fields
  vanishing on $Y$. Assume that $L\subseteq\Liealg(X,Y)$. If the fiber
  of $L$ over some $x_0\in X\setminus Y$ contains a generating set,
  then $X$ has the ADP relative to $Y$.
\end{theorem}

\begin{proof}
  Let $\{\partial_i\}$ be a finite set of vector fields in $L$ such
  that $\{\partial_i[x_0]\}$ is a generating set. Let now
  $\{\beta_j\}\subseteq \Aut(X,Y)$ be a finite collection of
  automorphisms fixing $x_0$ such that
  $\{\beta_j^*(\partial_i)[x_0]\}$ span the tangent space at $x_0$.
  Since change of coordinates does not change completeness of a vector
  field, for $\beta\in \Aut(X,Y)$, the finitely generated module
  $L_\beta=\beta^*(L)$ is again contained in $\Liealg(X,Y)$. By
  replacing $L$ with $\bigoplus_{j} L_{\beta_j}$, we can assume that
  $\{\partial_i[x_0]\}$ span the tangent space at $x_0$.

  We let $A_1=\{x\in X\setminus Y\mid \spann(\partial_i[x])\neq
  T_{x}X\}$.  We also let $A_1=\bigcup A_1^j$ be the decomposition of
  $A_1$ in irreducible components and we pick $x_j\in A_1^j$. Since
  $X$ is homogeneous with respect to $Y$, we can choose $\alpha_j\in
  \Aut(X,Y)$ sending $x_0$ to $x_j$. We also put
  $\alpha_0=\operatorname{Id}$. Let now
  $$A_2=\big\{x\in X\setminus
  Y\mid \spann\{\alpha_j^*(\partial_i)[x]\mid \forall i,j\}\neq
  T_{x}X\big\}\,.$$ %
  By construction $\dim A_1>\dim A_2$ and so we can proceed by
  induction on dimension to obtain a finite collection of
  automorphisms $\alpha_j\in \Aut(X,Y)$ such that the collection
  $\{\alpha_j^*(\partial_i)[x]\}$ span the tangent space at every
  point $x\in X\setminus Y$.

  We let $E=\bigoplus_{j} L_{\alpha_j}$. With the same argument as
  before, $E$ is a finitely generated $\CC[X]$-submodule of
  $\VFalg(X,Y)$ contained in $\Liealg(X,Y)$. By construction, we have
  that the fiber of $\widetilde{E}:=\VFalg(X,Y)\big/E$ at every $x\in
  X\setminus Y$ is trivial. Hence, the support of $\widetilde{E}$ is
  contained in $Y$.

  We define
  $$J=\operatorname{Ann}_{\CC[X]}\widetilde{E}:=\left\{f\in \CC[X]\mid
    fa=0\mbox{ for all } a\in \widetilde{E}\right\}\,.$$ %
  By construction $J\widetilde{E}=0$. This yields
  $J\VFalg(X,Y)\subseteq E$. Furthermore, by \cite[Ch. II Ex
  5.6]{Har77} we have that $V(J)\subseteq Y$. Recall that $I$ is the
  ideal of $Y$ and let $J'=J\cap I$ so that $V(J')=Y$. Let now $a_i$
  be a finite set of generators of $i$. Since
  $\operatorname{rad}(J')=I$, we have that there exists $\ell_i$ such
  that $a_i^{\ell_i}\in J$ for all $i$. Letting 
  $\ell=1+\sum_i(\ell_i-1)$ we obtain
  $$I^\ell\subseteq J'\subseteq J\quad\mbox{and so}\quad
  I^\ell\VFalg(X,Y)\subseteq J\VFalg(X,Y)\subseteq E\subseteq
  \Liealg(X,Y)\,.$$ %
  Hence the theorem follows.
\end{proof}

\section{The algebraic density property for affine toric varieties}

We first recall the basic facts from toric geometry that will be
needed in this section. They can be found in any text about toric
geometry such as \cite{Fu93,Oda88,CLS}.

Let $M$ and $N$ be mutually dual lattices of rank $n$ with duality
pairing $M\times N\rightarrow \ZZ$, where $(m,p)\mapsto \langle
m,p\rangle=p(m)$. We also let $M_\QQ=M\otimes_\ZZ \QQ$ and
$N_\QQ=N\otimes_\ZZ \QQ$. Letting $\TT$ be the algebraic torus
$\TT=\spec\CC[M]=N\otimes_\ZZ\CC^*$. A toric variety is a normal
variety endowed with an effective action of $\TT$ having an open
orbit. Since the $\TT$-action is effective, the open orbit is equal to
$\TT$.

It is well known that affine toric varieties can be described by means
of strongly convex polyhedral cones (pointed cones) in the vector
space $N_\QQ$. Indeed, let $\sigma$ be a pointed cone in $N_\QQ$, then
$X_\sigma=\spec \CC[\sigma^\vee\cap M]$ is an affine toric variety and
every affine toric variety arises this way. Here $\CC[\sigma^\vee\cap
M]$ is the semigroup algebra $\CC[\sigma^\vee\cap M]=\bigoplus_{m\in
  \sigma^\vee\cap M}\CC\chi^m$. In the following, we denote
$\sigma^\vee\cap M$ by $\sigma^\vee_ M$.

There is a one to one correspondence between the faces $\tau$ of the
cone $\sigma$ and the orbits $\OO(\tau)$ of the $\TT$-action on
$X_\sigma$ (usually called the Orbit-Cone correspondence). The
dimension of an orbit is given by $\dim \OO(\tau)=\rank N-\dim \tau$
and its closure is given by $\overline{\OO(\tau)}=\bigcup_\delta
\OO(\delta)$ where $\delta$ runs over all faces of $\sigma$ containing
$\tau$. The ideal $I(\tau)$ of an orbit closure $\overline{\OO(\tau)}$
is given by
$$I(\tau)=\bigoplus_{m\in \sigma^\vee_M\setminus
  \tau^\bot}\CC\chi^m\,$$ %
where $\tau^\bot\subseteq M_\QQ$ is the orthogonal of
$\tau$. Furthermore, the ideal of $X\setminus \TT$ is
$$I(X\setminus \TT)=\bigoplus_{m\in (\relint\sigma^\vee)\cap
  M}\CC\chi^m\,,$$ %
where $\relint$ denotes the relative interior.

As usual, we identify a ray $\rho\subseteq\sigma$ with its primitive
vector. The set of all the rays of $\sigma$ is denoted by
$\sigma(1)$. A cone $\sigma$ is called smooth if $\sigma(1)$ is part
of a basis of the lattice $N$. Let $\tau\subseteq \sigma$ be any
face. The orbit $\OO(\tau)$ is contained in $\Xreg$ if and only if
$\tau$ is smooth.

Let now $e\in M$ and $p\in N$. The linear map
$$\partial_{e,p}:\CC[M]\rightarrow\CC[M],\quad \chi^m\mapsto
\langle m,p\rangle\cdot\chi^{m+e}$$ is a homogeneous derivation of the
algebra $\CC[M]$ and so it is a homogeneous vector field on
$\TT=\spec\CC[M]$. By the exponential map, the tangent
space of $\TT=N\otimes_\ZZ\CC^*$ at the identity $\mathfrak{e}\in \TT$
is isomorphic to $N\otimes_\ZZ\CC$ and the evaluation of the vector
field $\partial_{e,p}$ at the smooth point $\mathfrak{e}$ is
$\partial_{e,p}[\mathfrak{e}]=p$.

Let $\sigma\subseteq N_\QQ$ be a pointed cone. The following
proposition gives a description of all the homogeneous vector fields
on $X_\sigma$. The first statement of the following result can be
found in \cite{Dem70}. For the convenience of the reader we provide a
short argument.

\begin{proposition}
  The homogeneous vector field $\partial_{e,p}$ on $\TT$ extends to a
  homogeneous vector field in $X_\sigma$ if and only if
  \begin{enumerate}[Type I:]
  \item $e\in\sigma^\vee_M$, or
  \item There exists $\rho_e\in \sigma(1)$ such that 
    \begin{enumerate}
    \item $p\in \ZZ \rho_{e}$,
    \item $\langle e,\rho_e\rangle=-1$, and
    \item $\langle e,\rho\rangle\geq 0$ for all
      $\rho\in\sigma(1)\setminus \{\rho_e\}$.
    \end{enumerate}
  \end{enumerate}
  Furthermore, $\partial_{e,p}$ is locally nilpotent if and only if it
  is of type II, and $\partial_{e,p}$ is semisimple if and only if it
  is of type I and $e=0$.
\end{proposition}
\begin{proof}

  The vector field $\partial_{e,p}$ extends to $X_\sigma$ if and only
  if $\partial_{e,p}(\CC[\sigma^\vee_M])\subseteq
  \CC[\sigma^\vee_M]$. Since $\CC[\sigma^\vee_M]$ is spanned by
  $\chi^m$ for all $m\in \sigma^\vee_M$, it is enough to show that
  $\partial_{e,p}(\chi^m)\in \CC[\sigma^\vee_M]$. In combinatorial
  terms, this corresponds to the condition:
  \begin{align}
    \label{eq:1}
    \mbox{For every } m\in \sigma^\vee_M\setminus p^\bot, 
    \mbox{ we have }
    \langle m+e, \rho\rangle\geq 0\mbox{ for all } \rho\in
    \sigma(1)\,.
  \end{align}
  
  Assume first that $p$ is not proportional to any $\rho\in
  \sigma(1)$. Then for every $\rho\in \sigma(1)$ there exists $m\in
  \sigma^\vee_M$ such that $\langle \rho,m\rangle=0$ and $\langle
  p,m\rangle\neq 0$. Hence, \eqref{eq:1} implies that $\langle
  \rho,e\rangle\geq 0$ and so $\partial_{e,p}$ is of type I.

  Assume now that there exists $\rho_e\in \sigma(1)$ such that $p\in
  \ZZ \rho_{e}$. With the same argument as above we can show that
  $\langle \rho,e\rangle\geq 0$ for all $\rho\in
  \sigma(1)\setminus\{\rho_e\}$. Let now $m\in \sigma^\vee_M$ such
  that $\langle \rho_e, m\rangle=1$. Then \eqref{eq:1} implies that
  $\langle \rho_e,m+e\rangle \geq 0$. This yields $\langle
  \rho_e,e\rangle \geq -1$. If $\langle \rho_e,e\rangle=-1$ then
  $\partial_{e,p}$ is of type II. If $\langle \rho_e,e\rangle>-1$ then
  $\langle \rho_e,e\rangle\geq 0$ and $\partial_{e,p}$ is of type I.
  
  \smallskip

  To prove the second assertion, we let $\partial=\partial_{e,p}$ be a
  homogeneous vector field. A straightforward computation shows that
  \begin{align}
    \label{eq:2}
    \partial^{\ell+1}(\chi^m)=\langle m+\ell e,p\rangle
    \cdot \partial^\ell(\chi^m)\cdot \chi^e\,.
  \end{align}

  Assume first that $\partial$ is of type I and that $e\in
  \sigma^\vee_M\setminus \{0\}$. If $\langle e,p\rangle\neq 0$ then
  \eqref{eq:2} yields
  $$\partial^{\ell}(\chi^e)=\ell!\cdot\langle
  e,p\rangle^\ell\cdot\chi^{\ell e}\neq 0\,,$$ and so $\partial$ is
  not locally finite since $\spann\{\chi^{k e}\mid k\in \ZZ_{\geq
    0}\}$ is not finite dimensional. If $\langle e,p\rangle=0$ then
  let $m\in \sigma_M^\vee$ be such that $\langle m,p\rangle\neq 0$. In
  this case \eqref{eq:2} implies
  $$\partial^{\ell}(\chi^m)=\langle
  m,p\rangle^\ell\cdot\chi^{m+(\ell-1) e}\neq 0\,,$$ and again $\partial$ is
  not locally finite with a similar argument.

  Assume now that $\partial$ is of type I and that $e=0$. The vector
  field $\partial$ is the infinitesimal generator of the algebraic
  $\CC^*$-action on $X_\sigma$ given by the $\ZZ$-grading on
  $\CC[\sigma_M^\vee]$ induced by the degree function
  $\deg(\chi^m)=\langle p,m\rangle$. Hence, the vector field
  $\partial$ is semisimple.

  Finally, assume that $\partial$ is of type II. For every $m\in
  \sigma^\vee_M$ we let $\ell=\langle m,\rho_e\rangle$. Now,
  $\partial_{e,p}$ is locally nilpotent since
  $\partial_{e,p}^{\ell+1}(\chi^m)=0$ by \eqref{eq:2}.
\end{proof}

In the following corollary, we give an explicit description of the
homogeneous complete vector fields on an affine toric
variety.

\begin{corollary} \label{cor-gl-int}
  The vector field $\partial_{e,p}$ is complete if and only
  if it is of type II, or it is of type I and $\langle e,p\rangle=0$.
\end{corollary}

\begin{proof}
  The vector fields of type II are locally nilpotent, hence
  complete. In the following, we assume that $\partial=\partial_{e,p}$
  is of type I. First, assume that $\langle e,p\rangle=0$. Then
  $\partial=\chi^e\cdot \partial_{0,p}$ and since $\chi^e$ belongs to
  the kernel of $\partial_{0,p}$, we have that $\partial$ is complete.

  Assume now that $\langle p,e\rangle\neq 0$. Let $I$ be the ideal of
  $X\setminus \TT$, i.e.,
  $$I=\bigoplus_{m\in \relint(\sigma^\vee)\cap M}\CC\chi^m\,.$$
  Since $e\in \sigma^\vee_M$, we have that $\partial(I)\subseteq I$.
  Hence, $X\setminus \TT$ is invariant by $\partial_{e,p}$ and so
  $\TT$ is also invariant by $\partial_{e,p}$. In the following, we
  show that $\partial$ is not complete when restricted to $\TT$. Since
  $\lambda\partial$, $\lambda\in \CC^*$ is complete if and only if
  $\partial$ is complete, we will assume that $p$ is a primitive
  vector in $N$ and $\langle e,p\rangle>0$.

  Without loss of generality, we choose mutually dual bases of $N$ and
  $M$ such that $p=(1,0,\ldots,0)$ and $e=(e_1,\ldots,e_n)$, with
  $e_1>0$ and $n=\rank N$. We will also denote $x_i=\chi^{\beta_i}$
  the standard coordinates of the torus $\TT$, where $\{\beta_i\mid
  i=1,\ldots, n\}$ is the base of $N$. In this coordinates, the vector
  field $\partial$ restricted to $\TT$ is given by
  $$\partial=x_1^{e_1+1}x_2^{e_2}\cdots
  x_n^{e_n}\frac{\partial}{\partial x_1}\,,$$ which is not complete on
  $\TT$ since $e_1>0$. Indeed the vector fields $x^n\partial/\partial x$ on $\C$ are not complete for $n\geq2$. 
\end{proof}

Remark that in Corollary~\ref{cor-gl-int} complete vector fields of
type I are extensions of complete vector fields on the big torus $\TT$
while complete vector fields of type II are locally nilpotent, hence
not complete in $\TT$. In the next lemma, we give a criterion for a
homogeneous vector field to vanish in an orbit closure.

\begin{lemma} \label{lm:vanish} Let $\partial_{e,p}$ be a non-zero
  homogeneous vector field on $X_\sigma$ and let $\tau\subseteq
  \sigma$ be a face. Then $\partial_{e,p}$ vanishes at the orbit
  closure $\overline{\OO(\tau)}$ if and only if

  \begin{enumerate}[Type I:]
  \item $p\in\Span\tau$ or $\langle e,\rho\rangle>0$ for some $\rho\in
    \tau(1)$.
  \item $\langle e,\rho\rangle>0$ for some $\rho\in \tau(1)$.
  \end{enumerate}
\end{lemma}

\begin{proof}
  The vector field $\partial_{e,p}$ does not vanish at the orbit
  closure $\overline{\OO(\tau)}$ if and only if
  $\partial_{e,p}\left(\CC[\sigma_M^\vee]\right)\not\subseteq
  I(\tau)$. In combinatorial terms this happens if and only if
  \begin{align}
    \label{eq:4}
    \mbox{there exists }m \in \sigma^\vee_M\setminus p^\bot \mbox{
      such that } \langle m+e,\rho\rangle=0 \mbox{ for all } \rho\in
    \tau(1)\,.
  \end{align}

  \noindent \textbf{Case of type I.} In this case, we have
  $e\in\sigma^\vee_M$ so $\langle m+e,\rho\rangle=0$ for all
  $\rho\in\tau(1)$ if and only if $\langle m,\rho\rangle=0$ and
  $\langle e,\rho\rangle=0$ for all $\rho\in\tau(1)$. This is the case
  if and only if $m\in\tau^\bot$ and $e\in\tau^\bot$. Such and
  $m\in\sigma^\vee_M\setminus p^\bot$ exists if and only if
  $\tau^\bot\not\subseteq p^\bot$, i.e., if and only if $p\notin \Span
  \tau$. Hence, we conclude that $\partial_{e,p}$ does not vanish at
  the orbit closure $\overline{\OO(\tau)}$ if and only if $p\notin
  \Span \tau$ and $\langle e,\rho\rangle=0$ for all $\rho\in \tau(1)$.
  
  \medskip\noindent \textbf{Case of type II.}  In this case we have
  that there exists $\rho_e\in \sigma(1)$ such that $p\in \ZZ
  \rho_e\setminus \{0\}$,
  $\langle e,\rho_e\rangle=-1$, and $\langle e,\rho\rangle\geq 0$ for
  all $\rho\in \sigma(1)\setminus \{\rho_e\}$.

  Assume first that $\rho_e\notin \tau(1)$. An argument similar to
  case I yields that $\partial_{e,p}$ does not vanish at the orbit
  closure $\overline{\OO(\tau)}$ if and only if $p\notin \Span \tau$
  and $\langle e,\rho\rangle=0$ for all $\rho\in \tau(1)$. Since
  $\rho_e\notin \tau(1)$, we have that $p\notin \Span \tau$ and so the
  vector field $\partial_{e,p}$ does not vanish at the orbit closure
  $\overline{\OO(\tau)}$ if and only if $\langle e,\rho\rangle=0$ for
  all $\rho\in \tau(1)$.

  Assume now that $\rho_e\in \tau(1)$. If there exists $\rho\in
  \tau(1)$ such that $\langle e,\rho\rangle>0$, then $\langle
  m+e,\rho\rangle>0$ for all $m\in \sigma^\vee_M$ and so
  $\partial_{e,p}$ vanishes at the orbit $\overline{\OO(\tau)}$ by
  \eqref{eq:4}. Assume $\langle e,\rho\rangle=0$ for all $\rho\in
  \tau(1)\setminus \{\rho_e\}$ and let $m\in \sigma^\vee_M$ be such
  that $\langle m,\rho_e\rangle=1$ and $\langle m,\rho\rangle=0$ for
  all $\rho\in \tau(1)\setminus \{\rho_e\}$. We have $\langle
  m,\rho_e\rangle\neq 0$ so $m\notin p^\bot$ and $\langle
  m+e,\rho\rangle=0$ for all $\rho\in \tau(1)$. By \eqref{eq:4}, we
  conclude that $\partial_{e,p}$ does not vanish at the orbit closure
  $\overline{\OO(\tau)}$.
\end{proof}

\begin{remark}
  The degree of a homogeneous locally nilpotent vector fields (of type
  II) is called a root of $\sigma$. The set of all roots of $\sigma$
  is denoted by $\RR(\sigma)$. For a root $e\in\RR(\sigma)$, the ray
  $\rho_e$ is called the distinguished ray of $e$ and the $\GA$-action
  generated by the homogeneous locally nilpotent vector field
  $\partial_{e,\rho_e}$ is denoted by $H_e$.
\end{remark}

In order to show the ADP for toric varieties, we need to show that
$X_\sigma$ is homogeneous with respect to some closed subvariety
$Y$. In \cite{AKZ10}, the authors prove that $X_\sigma$ is homogeneous
with respect to $\Xsing_\sigma$. In fact, they show that the group of
special automorphisms acts infinite-transitively with respect to
$\Xsing_\sigma$. In the following, we will show how their methods can
be applied to show that $X_\sigma$ is homogeneous with respect to any
$\TT$-invariant closed subvariety $Y$.

\begin{proposition} \label{rel-hom}
  Let $\sigma\subseteq N_\QQ$ be a pointed cone and let $Y$ be any
  $\TT$-invariant closed subvariety of $X_\sigma$ containing
  $\Xsing_\sigma$. Then $X_\sigma$ is homogeneous relative to $Y$.
\end{proposition}

\begin{proof}
  Using the $\TT$-action and the Orbit-Cone correspondence, to prove
  the theorem it is enough to find, for every orbit $\OO(\tau)$ in
  $\Xreg_\sigma$ different from the open orbit, an automorphism that
  \begin{enumerate}[$(i)$]
  \item sends a point $x$ in $\OO(\tau)$ into an
  orbit of higher dimension, and
\item leaves stable every orbit not containing $\OO(\tau)$ in its
  closure.
  \end{enumerate}
   
  Let $\rho_1,\ldots,\rho_\ell$ be the rays of $\tau$.  In
  \cite[Lemma~2.3]{AKZ10} and its proof, the authors show that for
  every smooth orbit $\OO(\tau)$ there exists a root $e\in \RR(\sigma)$
  such that
  \begin{align}
    \label{eq:3}
    \langle \rho_1,e\rangle=-1,\ \langle \rho_2,e\rangle=\ldots=
    \langle \rho_\ell,e\rangle=0,\mbox{ and } \langle \rho,e\rangle>0\
    \mbox{ for all rays }\rho\not\notin\tau(1)\,.
  \end{align}
  Furthermore, they show that a generic automorphism $\alpha$ in the
  $\GA$-action $H_e$ corresponding to the root $e$ satisfies $(i)$.

  Let $\OO(\delta)$ be any orbit that does not contain $\OO(\tau)$ in
  its closure. In combinatorial terms, this means that $\delta$ is a
  face of $\sigma$ that is not contained in $\tau$.  We claim that
  $H_e$ leaves $\overline{\OO(\delta)}$ point-wise invariant and so
  $\alpha$ satisfies $(ii)$ which proves the proposition.
  
  In terms of the vector field $\partial_{e,\rho_e}$, our claim is
  equivalent to $\partial_{e,\rho_e}$ vanishes at
  $\overline{\OO(\delta)}$.  Since $\delta$ is not contained in $\tau$
  there exists a ray $\rho$ of $\delta$ that is not a ray of
  $\tau$. By \eqref{eq:3} we have $\langle e,\rho\rangle>0$. Now the
  claim follows from Lemma \ref{lm:vanish}.
\end{proof}

For our next theorem we need the following lemma that follows by
direct computation.

\begin{lemma} \label{commutator} %
  Let $\partial_{e_1,p_1}$ and $\partial_{e_2,p_2}$ be two homogeneous
  vector fields. Then
  $\left[\partial_{e_1,p_1},\partial_{e_2,p_2}\right]=\partial_{e,p}$,
  where $p=p_1(e_2)\cdot p_2-p_2(e_1)\cdot p_1$ and $e=e_1+e_2$.
\end{lemma}

\begin{theorem} \label{finalthm}%
  Let $X$ be a affine toric variety of dimension at least two and let
  $Y$ be a $\TT$-invariant closed subvariety of $X$ containing
  $\Xsing$. Then $X$ has the ADP relative to $Y$ if and only if
  $X\setminus Y\neq \TT$.
\end{theorem}

\begin{proof}
  Let $X=X_\sigma$ be the toric variety given by the pointed cone
  $\sigma\in N_\QQ$ and let $X_\sigma\setminus Y\neq \TT$. There is at
  least one codimension one $\TT$-orbit not contained in $Y$. Assume
  it is $\OO(\rho_1)$ for some ray $\rho_1\in \sigma(1)$. Let $e_1$ be
  a root with $\rho_1$ as distinguished ray. By \eqref{eq:3}, we can
  assume that $\langle e_1,\rho\rangle>0$ for all $\rho\in
  \sigma(1)\setminus \{\rho_1\}$. By Lemma~\ref{lm:vanish}, the
  locally nilpotent vector field $\partial_{e_1,\rho_1}$ vanishes at
  $Y$ and so $\partial_{e_1,\rho_1}\in \VFalg(X_\sigma,Y)$.

  Letting $e_2,e_3\in \relint(\sigma^\vee)\cap M$ be such
  that $e_3=e_1+e_2$, we let
  $$L=\Span\left\{\partial_{e,p}\mid p\in N, e\in e_3+\sigma^\vee_M\right\}\,.$$

  The set $L$ is contained in $\VFalg(X_\sigma,Y)$ since
  $\partial_{e,p}\in L$ vanishes in $X_\sigma\setminus \TT$. In fact,
  $L$ is a submodule of $\VFalg(X_\sigma,Y)$ since for every $m\in
  \sigma^\vee_M$ and every $\partial_{e,p}\in L$, we have
  $\chi^m\partial_{e,p}=\partial_{e+m,p}\in L$.  Furthermore, the
  fiber over the identity $\mathfrak{e}\in \TT\subseteq X_\sigma$ is
  given by
  \begin{align} \label{spanning}
    L_\mathfrak{e}=\Span
    \{\partial_{e,p}[\mathfrak{e}]\mid \partial_{e,p}\in
    L\}=\Span\{p\mid \partial_{e,p}\in L\}=N\otimes_\ZZ
    \CC=T_\mathfrak{e}X_\sigma\,,
  \end{align} %
  and so $L_\mathfrak{e}$ contains a generating set.  We claim that
  $L\subseteq \Liealg(X_\sigma,Y)$.  Hence $X_\sigma$ has the ADP
  relative to $Y$ by Theorem~\ref{thm} and Proposition~\ref{rel-hom}.

  By Corollary~\ref{cor-gl-int}, the vector field $\partial_{e,p}$ is
  complete if $\langle e,p\rangle=0$. Hence, to prove our
  claim it is enough to show that for every $e\in e_3+\sigma^\vee_M$,
  there exists $p\in N$ such that $\langle e,p\rangle\neq 0$ and
  $\partial_{e,p}\in \Liealg(X_\sigma,Y)$.

  Indeed, let $e_4=e-e_1$ and choose $p_4$ be such that $\langle
  e_4,p_4\rangle=0$ and $\langle e_1,p_4\rangle\neq 0$ which implies
  that $\partial_{e_4,p_4}$ belongs to $\Liealg(X_\sigma,Y)$. This is
  possible since $e_4$ lies in $\relint\sigma^\vee$ and $e_1$ is a
  root of $\sigma^\vee$. By Lemma~\ref{commutator} we have
  $$\left[\partial_{e_1,\rho_1},\partial_{e_4,p_4}\right]=
  \partial_{e,p}\quad\mbox{where}\quad p=\rho_1(e_4)\cdot
  p_4-p_4(e_1)\cdot \rho_1\,.$$ %
  A routine computation shows that
  $$\langle e,p\rangle=\
  \langle e,\rho_1(e_4)\cdot p_4-p_4(e_1)\cdot \rho_1\rangle=\langle
  e_1,p_4\rangle\neq 0\,,$$ proving the claim.

  \medskip

  Assume now that $X\setminus Y= \TT$. The converse of the theorem
  follows from the fact that for all affine toric varieties $X$ and
  all $\ell\in \ZZ_{>0}$ there is a vector field $\partial\in
  I^\ell\VFalg(X,X\setminus \TT)\setminus \Liealg(X,X\setminus \TT)$,
  where $I=I(X\setminus \TT)$. Indeed, Anders\'en \cite{Andersen00}
  proved that any complete algebraic vector field on $\TT$ does
  preserve the Haar form
  $$\omega=\frac{dx_1}{x_1}\wedge\ldots\wedge \frac{dx_n}{x_n}\,.$$
  Thus if we find $\partial$ in $I^\ell\VFalg(X,X\setminus \TT)$ whose
  restriction to $\TT$ does not preserve $\omega$ we are done.

  After a change of coordinates one can assume that
  $(1,0,\ldots,0)\in\relint\sigma^\vee$. Then
  $\partial=x_1^N\frac{\partial}{\partial x_1}$ is a regular vector
  field on $X$ contained in $I^\ell\VFalg(X,X\setminus \TT)$ for $N$
  big enough which does not preserve $\omega$.
\end{proof}

\begin{remark}
  L\'arusson proved in \cite{La11,For13} that all smooth toric
  varieties are Oka-Forstneri\v{c} manifolds, however it is still
  unknown if they are elliptic, see \cite{Forst11,Ku14} for
  definitions.  The proof of Theorem~\ref{finalthm} can be adapted to
  prove the following: every smooth quasi-affine toric variety is
  elliptic (and thus an Oka-Forstneri\v{c} manifold).  Indeed, the
  torus $\TT$ is well known to be elliptic. Let $X_0$ be a smooth
  quasi-affine toric variety different from $\TT$. Let also $X$ be an
  affine toric variety such that $X_0\subseteq X$ is an equivariant
  open embedding and let $Y=X\setminus X_0$. Now,
  Proposition~\ref{rel-hom} and \eqref{spanning} imply that $X_0$ is
  elliptic \cite[Example~5.5.13~(B)]{Forst11}.
\end{remark}

\section{Classification of complete vector fields on affine toric
  surfaces}

In this section we classify all complete algebraic vector
fields on a given affine toric surface $X_\sigma$. The classification
works essentially the same as the
classification of complete vector fields on $\C^2$ done by
Brunella \cite{Br3}.

From now on we will use the fact that each affine toric surface
different from $\C^*\times\C$ or $\C^*\times\C^*$ can be seen as the
quotient of $\C^2$ by the action of a cyclic group. Let $d$ be the
order of the group and let $e$ be a co-prime number $0<e<d$ and
consider the action of $\Z_d$ given by $\zeta\cdot(u,v) = (\zeta u,
\zeta^e v)$ where $\zeta$ is a primitive $d$-th root of unity. We
obtain the projection $\pi: \C^2 \rightarrow \C^2/\Z_d =: V_{d,e}$
onto our toric surface which is a ramified covering of $V_{d,e}$
ramified only over the unique singular point. Certainly each vector
field on $X$ pulls back to an invariant vector field of $\C^2$ by
using the fiber-wise isomorphism $D\pi$ on the tangent space. A
complete vector field on $V_{d,e}$ will pull back to an invariant
complete vector field on $\C^2$.

\begin{definition} 
  Let $f: \C^2 \rightarrow \C$ be a regular function on $\C^2$. The
  function $f$ is called $\Z_d$-preserved if the fibers of $f$ are
  sent to fibers of $f$ by the $\Z_d$-action. It is called
  $\Z_d$-homogeneous of degree $[i]\in\Z_d$ if $\zeta^*f(u,v)=
  f(\zeta\cdot(u,v)) = \zeta^if(u,v)$ for all $(u,v)\in\C^2$. Let
  $A_{[i]}$ denote the space of $\Z_d$-homogeneous polynomials of
  degree $[i]$ then we obtain a decomposition of the ring of regular
  functions on $\C^2$ into $\Z_d$-homogeneous parts $\C[u,v] = A_{[0]}
  \oplus \ldots \oplus A_{[d-1]}$. In particular $A_{[0]}$ is the ring
  of invariant functions $\C[u,v]^{\Z_d}=\C[V_{d,e}]$.
\end{definition}

It is clear from the definition that $A_{[i]}$ is spanned by all
monomials $u^mv^n$ with $[m+en]=[i]\in\Z_d$. Clearly invariant vector
fields are of the form $f\partial/\partial u + g\partial/\partial v$
with $f\in A_{[1]}$ and $g\in A_{[e]}$. Moreover we have the following
easy lemma:

\begin{lemma} \label{lemhomog} Let $f: \C^2 \rightarrow \C$ be a
  regular function then the following are equivalent:
  \begin{enumerate}
  \item $f$ is $\Z_d$-homogeneous,
  \item $f$ is $\Z_d$-preserved with $f(0,0)=0$,
  \item $f^{-1}(0)$ is $\Z_d$-invariant.
  \end{enumerate}

\end{lemma}
\begin{proof}
  (1) implies (2) since if $f$ is constant on a curve then also
  $\zeta^i \cdot f$ is constant and $f(0,0)=0$ follows directly from
  the homogeneity. The fiber $f^{-1}(0)$ contains the $\Z_d$-fixed
  point $(0,0)$ thus (3) follows from (2). If the zero fibers of $f$
  and $\zeta^*f$ coincide then we have that $\zeta^*f(u,v) = a\cdot
  f(u,v)$ for some $a\in\C^*$. By $f(u,v)=\zeta^{d*}f(u,v)=a^d f(u,v)$
  we see that $a$ is a $d$-th root of unity and thus (3) implies (1).
\end{proof}

The following lemma is the crucial step in the classification of
invariant complete algebraic vector fields and hence of
complete algebraic vector fields on the toric variety
$V_{d,e}$. Recall that a rational first integral of a vector field is a rational function such that its fibers are tangential to the vector field.
\begin{lemma}\label{lemhomfib}
  Let $\partial$ be a $\Z_d$-invariant complete algebraic vector field
  on $\C^2$ then $\partial$ preserves either a $\Z_d$-homogeneous
  fibration $f: \C^2 \rightarrow \C$ with general fibers $\C$ or
  $\C^*$ or $\partial$ has a reduced rational first integral $g:\C^2
  \dashrightarrow \C$.
\end{lemma}
\begin{proof}
  By \cite{Br3} there is fibration $f:\C^2\rightarrow\C$ with $\C$ or
  $\C^*$ fibers which is preserved by the flow $\varphi^t$ of
  $\partial$. We may assume that $f(0,0)=0$. If $f$ is
  $\Z_d$-homogeneous then we are done. If $f$ is not
  $\Z_d$-homogeneous then we construct a rational first integral. The
  map $\varphi^t$ acts by multiplication with some $a_t$ on the set of
  fibers of $f$ parametrized by $\C$ so we have
  $f(\varphi^t(u,v))=a_tf(u,v)$ (indeed $(0,0)$ is a fixed point of
  $\varphi^t$). Since $\partial$ is invariant the same holds true for
  $g(u,v)=f(\zeta\cdot(u,v))$ and hence the rational map $f/g$ is a
  rational first integral for $\partial$. By Stein factorization
  $\partial$ has a reduced first integral. Recall that every rational
  function $\C^2\dashrightarrow \P^1$ can be decomposed into
  $F\circ\tilde f\C^2\dashrightarrow \P^1\rightarrow \P^1$ such that
  $\tilde f$ has connected regular fibers, or equivalently is
  reduced. This factorization is called Stein factorization.
\end{proof}

The next step will be the classification of $\Z_d$-homogeneous
fibrations with general fibers $\C$ or $\C^*$ and rational first
integrals for invariant vector fields. The classification will be done
up to equivariant automorphisms of $\C^2$ which will lead to a
classification of the vector fields on $V_{d,e}$ up to automorphism of
$V_{d,e}$ since equivariant automorphisms clearly project down to
automorphims of the quotient. Equivariant automorphisms of $\C^2$ are
given by invertible maps $(u,v)\mapsto (p(u,v),q(u,v))$ with $p\in
A_{[1]}$ and $q\in A_{[e]}$.

First we establish an equivariant version of the Abhyankar-Moh
Theorem. We provide a proof using the classical verion of the
theorem. See \cite{ArZa} for a different proof.

\begin{lemma} \label{lemabhy} %
  Let $\C\cong L\subset \C^2$ be a line which is invariant by the
  group action. Then there is an equivariant automorphism of $\C^2$
  mapping $L$ to $\lbrace u =0\rbrace$ or $\lbrace
  v=0\rbrace$. Moreover a cross of two invariant lines can be mapped
  to $\lbrace uv=0\rbrace$.
\end{lemma}

\begin{proof}
  By the classical Abhyankar-Moh Theorem we know that $L$ is given by
  a polynomial $p$ which is a component of an automorphism of $\C^2$.
  In order to find the other component of the automorphism we have to
  find an invariant section of the trivial line bundle given by
  $p$. We start with an arbitrary trivialization and get an invariant
  section taking the average over images of the zero section by the
  group action. Each image is another section because the action sends
  fibers of $p$ to fibers of $p$ since the zero fiber is invariant. We
  denote the polynomial giving this invariant section by $q$. The map
  given by $(p,q)$ is an automorphism of $\C^2$ since it is the
  composition of the trivialization we started with and the map
  $(u,v)\mapsto (u,v -s(u))$ where $s$ is the invariant
  section. Because the zero sets of $p$ and $q$ are invariant they are
  $\Z_d$-homogeneous by Lemma \ref{lemhomog} and since they are the
  two components of an automorphism their homogeneity degrees coincide
  with $[1]$ and $[e]$ so either $(p,q)$ or $(q,p)$ is an equivariant
  automorphism and the claim follows. The second statement is trivial
  since there we already have an invariant section by assumption.
\end{proof}
We get the following corollary as an immediate consequence, see also
\cite{FlKaZa}.
\begin{corollary} \label{c-fiber}
 Let $f:\C^2\rightarrow\C$ be a $\Z_d$-homogeneous fibration with $\C$ fibers and $f(0,0)=0$ then $f(u,v)=u$ or $f(u,v)=v$ up to equivariant automorphism of
$\C^2$.
\end{corollary}

For the classification of $\Z_d$-homogeous fibration with $\C^*$
fibers we first state the non-equivariant version used in \cite{Br3},
see also \cite{Su}.

\begin{lemma}\label{lemcstarfib}
  Let $f:\C^2\rightarrow\C$ be a fibration with $\C^*$ fibers then
  $f(x,y)$ has one special fiber (say $f^{-1}(0)$) and it is
  isomorphic to $\C\cup\C^*$ or $\lbrace xy=0\rbrace$ and $f$ is up to
  automorphism of $\C^2$ of the form $f(x,y)=x^m(x^ly + p(x))^n$ or
  $f(x,y)=x^my^n$ for coprime $m,n\in\N$, $\deg p<l\geq1$ and
  $p(0)\neq 0$.
\end{lemma}
The equivariant version of this lemma is given by the two following
lemmas.
\begin{lemma} \label{cstar-fiber-1}
  Let $f:\C^2\rightarrow \C$ be a $\Z_d$-homogeneous fibration with
  $\C^*$ fibers and $f^{-1}(0)\cong\C\cup\C^*$ then there are coprime
  $m,n\in\N$ and an invariant polynomial $p$ with $\deg p <l\geq1$ and
  $p(0)\neq 0$ such that up to equivariant automorphism
  $f(u,v)=u^m(u^lv + p(u))^n$ with $[l+e]=[0]$ or
  $f(u,v)=v^m(v^lu+p(v))^n$ with $[1+le]=[0]$.
\end{lemma}

\begin{proof}
  By Lemma~\ref{lemcstarfib} we know that there exists a not necessary
  equivariant automorphism $(x(u,v),y(u,v))$ such that $f(x,y)$ is as
  in Lemma~\ref{lemcstarfib}. Clearly, the curve $\C\cong C \subset
  f^{-1}(0)$ is invariant by the group action since it is the only
  fiber component isomophic to $\C$. By Lemma \ref{lemabhy} we may
  assume the $C = \lbrace u=0\rbrace$ or $C=\lbrace v=0\rbrace$. In
  the first case this implies that, up to equivariant automorphism,
  $x(u,v)=au$ and $y(u,v)=bv + q(u)$ for some $a,b\in\C^*$ and
  $q\in\C[u]$ and hence $f$ is of the form $(au)^m((au)^l(bv+q(u)) +
  p(u))^n$ with $\deg p<l$. Since $f$ is $\Z_d$-homogeneous we have
  $q\in A_{[e]}$ and $p\in A_{[l+e]}$ hence the map $(x(u,v),y(u,v))$
  was equivariant after all and $f$ has the desired standard form up
  to equivariant automorphism. The equality $[l+e]=[0]$ follows from
  the fact that $p(0)\neq 0$. The case $C=\lbrace v=0\rbrace$ leads
  similarly to the second possibility.
\end{proof}

\begin{lemma} \label{cstar-fiber-2} Let $f:\C^2\rightarrow \C$ be a
  $\Z_d$-homogeneous fibration with $\C^*$ fibers and
  $f^{-1}(0)\cong\lbrace uv=0\rbrace$ then there are coprime
  $m,n\in\N$ such that $f(u,v)=u^mv^n$ up to equivariant
  automorphism. If $d$ is divisible by 4 (say $d=4d'$) and $e=2d'+1$
  then $f$ can also be of the form $f(u,v)=u^2-v^2$.
\end{lemma}
\begin{proof}
  By Lemma \ref{lemcstarfib} there is an automorphism
  $(x(u,v),y(u,v))$ such that $f=x^my^n$. Clearly the 0-fiber $\lbrace
  x(u,v)=0\rbrace \cup \lbrace y(u,v)=0 \rbrace$ is invariant by the
  group action. If the two lines are invariant themselves then by
  Lemma \ref{lemabhy} we may assume that they coincide with $\lbrace
  uv=0\rbrace$ and hence we may assume $x(u,v)=au$ and $y(u,v)=bv$ for
  some $a,b\in\C^*$ and we are done. If the two lines are interchanged
  by the group action then we have $d=2d_0$ is even and
  $$ x(\zeta\cdot(u,v))= ay(u,v) \quad \mathrm{and}\quad 
  y(\zeta\cdot(u,v)=bx(u,v)$$ %
  for some $a,b \in \C^*$. After rescaling we may assume that
  $a=b$. The fibration $f=x^my^n$ is $\Z_d$-homogeneous so
  $x^my^n=const \cdot y^mx^n$ and hence $m=n=1$. Moreover we have
  $x(u,v)=x(\zeta^d\cdot(u,v))=a^dx(u,v)$ and hence $a=\zeta^i$ for
  some $i$. We see that the maps $P_\pm(u,v)=x(u,v)\pm y(u,v)$ are
  $\Z_d$-homogeneous and since they are the components of an
  automorphism of $\C^2$ we may assume that the functions $P_\pm$
  coincides with the functions $u$ and $v$. Altogether we have
  $\frac{1}{4}f(u,v)=(u+v)(u-v)=u^2-v^2$ which is $\Z_d$-homogeneous
  only if $2e=2$ or $2e=2d_0+2$. In the first case $(x(u,v),y(u,v))$ is
  already equivariant so only the latter case remains. Since $d$ is even and thus $e=d_0+1$ is odd we have that $d_0=2d'$ is even.
\end{proof}

\begin{lemma}\label{lemfirstint}
  Let $f: \C^2 \dashrightarrow \P^1$ be a reduced rational first
  integral of an invariant complete vector field $\partial$ on $\C^2$
  then up to equivariant automorphism of $\C^2$ and M\"obius transform
  of $\P^1$ the rational function $f$ is a $\Z_d$-homogeneous
  polynomial with $\C$ or $\C^*$ fibers or there are coprime
  $m,n\in\N$ such that $f(u,v)=u^m/v^n$.
\end{lemma}
\begin{proof}
  A general fiber of $f$ is an orbit closure of the flow of
  $\partial$. Since $\partial$ is invariant the set of orbits is preserved
  by the $\Z_d$-action hence general fibers of $f$ are mapped to
  general fibers of $f$ by the action and the action induces a
  $\Z_d$-action on the base $\P^1$. Altogether this means that $f$ is
  $\Z_d$-preserved. If $f$ is not surjective then $f$ can be seen as a
  polynomial which is $\Z_d$-homogeneous by Lemma \ref{lemhomog} and
  has general fibers isomomorphic $\C$ or $\C^*$ since they are orbit
  closures.

  Now consider the case $f$ surjective.  As mentioned in \cite{Br3}
  and \cite{Su} such a first integral is always of the form
  $f=x^m/y^n$ for some automorphism $(x(u,v),y(u,v))$. The
  $\Z_d$-action on the base $\P^1$ is either trivial (and hence $f$ is
  $\Z_d$-invariant) or it has exactly two fixed points (so two fibers
  of $f$ are $\Z_d$-invariant). In both cases there are two invariant
  fibers intersecting transversally (say the $0$- and the
  $\infty$-fiber). Indeed if $m=n=1$ all fibers intersect transversely
  and if $m\neq n$ all but one fiber intersect pairwise tangentially
  so this fiber is clearly invariant and it intersects all other
  fibers transversally. By Lemma \ref{lemabhy} we may assume that
  these two fibers coincides with $\lbrace u = 0\rbrace$ and $\lbrace
  v = 0\rbrace$ and hence $x(u,v)=au$ and $y(u,v)=bv$ or vice versa.
\end{proof}

\begin{theorem} \label{thmlist} Let $\partial$ be a complete algebraic
  vector field on $\C^2$ which is invariant by the group action given
  by $\zeta\cdot(u,v) = (\zeta u,\zeta^e v)$ where $\zeta$ is a
  primitive $n$-th root of unity and $0<e<d$ coprime numbers. Then
  $\partial$ has, up to equivariant automorphism of $\C^2$, one of the
  forms in the following list. 
\begin{enumerate}\itemsep8pt
\item [] \item \begin{enumerate} \itemsep8pt
   \item $\displaystyle{ \partial=au\frac {\partial}{\partial u} +
       ((A(u^d)v+B(u^e))\frac{\partial}{\partial v}}$
	\item $\displaystyle{ \partial=av\frac {\partial}{\partial v} + ((A(v^d)u+B(v^{e'}))\frac{\partial}{\partial u}}$
	\end{enumerate}
      \item[]with $a\in\C$, $0<e'<d$ such that $[ee']=[1] \in \Z_d$
        and $A,B\in \C[t]$.
\item \begin{enumerate}\itemsep8pt 
  \item $\displaystyle{\partial=av\frac{\partial}{\partial v} +
      A(u^mv^n)\left[nu\frac{\partial}{\partial u} -
        mv\frac{\partial}{\partial v}\right]}$
\item If $d=4d'$ and $e=2d'+1$ then we also have
\item[] $\displaystyle{\partial=a(u+v)\left(\frac{\partial}{\partial u}
      + \frac{\partial}{\partial v}\right) +
    A((u^2-v^2)^{2d'})\left[u\frac{\partial}{\partial v} +
      v\frac{\partial}{\partial u}\right]}$
 \end{enumerate}
\item[]with $a\in\C$, $m,n\in\N$ with $[m+en]=[0]$ and
  $A\in\C[t]$ .
     
\item There are $a\in\C$, $m,n,l\in \N$ with $[m]=[0]$, $p\in
  A_{[0]}$, $\deg p< l$, $p(0)\neq 0$ and $A\in\C[t]$ with the
  property that \[A(x^m(x^ly+p(x))^n)\cdot(mp(x)+nxp'(x))-ap(x) \in
  x^l\cdot\C[x,y]\] such that
\item[] \begin{enumerate} \itemsep8pt
  \item $\displaystyle{ \partial= a\left(v+\frac{p(u)}{u^l}\right)
      \frac{\partial}{\partial v}+}$
  \item[] \quad $\displaystyle{A(u^m(u^lv +
      p(u))^n)\cdot\left[nu\frac{\partial}{\partial u} - \left((m+nl)v
          + \frac{mp(u)+nup'(u)}{u^l}\right)\frac{\partial}{\partial
          v}\right]}$
  \item[] with $[l+e]=0$.
  \item $\displaystyle{ \partial= a\left(u+\frac{p(v)}{v^l}\right)
      \frac{\partial}{\partial u} + }$
  \item[] \quad $\displaystyle{A(v^m(v^lu +
      p(v))^n)\cdot\left[nv\frac{\partial}{\partial v} - \left((m+nl)u
          + \frac{mp(v)+nvp'(v)}{v^l}\right)\frac{\partial}{\partial
          u}\right]}$
  \item[] with $[1+le]=0$.
\end{enumerate}
\end{enumerate}

\end{theorem}
\begin{proof}
  By Theorem \ref{lemhomfib} we know that the flow $\partial$ preserves
  fibers of a $\Z_d$-homogeneous $\C$- or $\C^*$-polynomial (which are
  described in Corollary~\ref{c-fiber} and Lemmas~\ref{cstar-fiber-1}
  and \ref{cstar-fiber-2}) or it has a rational first integral (which
  may be assumed to be of the form $u^m/v^n$ by Lemma
  \ref{lemfirstint}). Once we have a polynomial that is preserved by
  the flow we can check in Proposition 2 in \cite{Br3} how the vector
  field looks like. Since the vector fields need to be
  $\Z_d$-invariant some extra conditions are required. In the case of
  the rational first integral we have $\partial=nu\partial/\partial u +
  mv\partial/\partial v$ which is already in the list.
\end{proof}

\section{The strong algebraic density property for affine toric
  surfaces}

First we give a new concept of the ADP which was first introduced in
\cite{KaKu14}.

\begin{definition}
  Let $\Gamma$ be a group acting on an smooth affine algebraic variety
  $X$. Then $X$ has $\Gamma$-ADP if the Lie algebra of all
  $\Gamma$-invariant algebraic vector fields coincides with the Lie
  algebra generated by all $\Gamma$-invariant complete algebraic
  vector fields.
\end{definition}
As in the section above let $d,e\in\Z$ be two coprime numbers with
$0<e<d$ and let $\zeta$ be a primitive $d$-th root of unity. Consider
again the $\Z_d$-action on $\C^2$ given by $\zeta\cdot(u,v)=(\zeta u,
\zeta^e v)$. Moreover let $e'$ be the unique integer with $0<e'<d$ and
$ee'=1$ mod $d$.  It is clear that:
\begin{proposition} \label{surf-ADP}
  $V_{d,e}$ has the strong ADP if and only if $\C^2$ has the
  $\Z_d$-ADP.
\end{proposition}
Let us introduce the following subsets of $\ZZ^2$
\begin{eqnarray*}
  I & = & \lbrace (i,j)\in\ZZ_{\geq 0}^2: \ i+ej = 0 \ \mathrm{mod} \ d \rbrace,\\ 
  J & = & \lbrace (i,j)\in I\setminus\lbrace(0,0)\rbrace: \ i<e \ \mathrm{and} \ j<e' \rbrace \subset I,\\
\end{eqnarray*}

\begin{lemma}
  $|J| \leq1 \Leftrightarrow e \ \vert \ d+1$.
\end{lemma}
\begin{proof}
  If $e=1$ then also $e'=1$ and thus $J=\emptyset$. If $e,e'>1$ then $
  |J| \geq 1$ since $(e-1,e'-1)\in J$. Assume $ee'=d+1$, $i<e$ and
  $j<e'$, then we have $i+je<e+d<2d$ and the equality $[i+je]=[0] \in
  \Z_d$ implies $i+je=d$. Similarly we get $ie'+j=d$ and thus there is
  a unique solution for $(i,j)$ and hence $|J| =1$. If $ee'\geq 2d+1$
  then we get another solution of $[i+je]=[0]$ in $J$. Indeed, choose
  $l\in\N$ such that $0<d-le<e$ then $(d-le,l)\neq(e-1,e'-1)$ lies in
  $J$, since $le<d$ implies $0<l<e'-1$.
\end{proof}

Let us introduce the following notation:

\begin{eqnarray*}
 \mathrm{VF}^{(i,j)} & = & \left\lbrace u^i v^j \left(au\frac{\partial}{\partial u} + bv\frac{\partial}{\partial v}\right): \ a,b\in \C \right\rbrace,\\
 \mathrm{CVF}^{(i,j)} & = & \left\lbrace au^i v^j \left(ju\frac{\partial}{\partial u} - iv\frac{\partial}{\partial v}\right): \ a\in \C \right\rbrace
\subset\mathrm{VF}^{(i,j)}, \\
\mathrm{LND}_u^k &=& \left\lbrace av^{ke'}\frac{\partial}{\partial u}: \ a\in\C\right\rbrace,\\
\mathrm{LND}_v^k&  = &\left\lbrace au^{ke}\frac{\partial}{\partial v}: \ a\in\C\right\rbrace.\\
\end{eqnarray*}

Remark that $\mathrm{CVF}^{(i,j)}$ corresponds to the subset of
complete vector fields in $\mathrm{VF}^{(i,j)}$ by
Corollary~\ref{cor-gl-int}. We have the decomposition of
$\Z_d$-invariant vector fields in homogeneous vector fields given by:
\[
\mathrm{VF}_{\mathrm{alg}}^{\Z_d}(\C^2) = \bigoplus_{(i,j)\in
  I}\mathrm{VF}^{(i,j)} \oplus
\bigoplus_{k\in\N}\left(\mathrm{LND}_u^k \oplus
  \mathrm{LND}_v^k\right).
\]
We define the subspace $S$ of
$\mathrm{VF}_{\mathrm{alg}}^{\Z_d}(\C^2)$.
\[
 S = \bigoplus_{(i,j)\in J}\mathrm{CVF}^{(i,j)} \oplus \bigoplus_{(i,j)\in I\setminus J}\mathrm{VF}^{(i,j)} \oplus
\bigoplus_{k\in\N}\left(\mathrm{LND}_u^k \oplus \mathrm{LND}_v^k\right)\,.
\]

The following is our main result in this section.

\begin{theorem} \label{thminvliealg} For the Lie algebra
  $\mathrm{Lie}_{\mathrm{alg}}^{\Z_d}(\C^2)$ generated by all
  $\Z_d$-invariant complete algebraic vector fields on $\C^2$ we have:
\[
 \mathrm{Lie}_{\mathrm{alg}}^{\Z_d}(\C^2)=
\begin{cases}
  S & e = e'\\
  S \oplus \langle \partial \rangle & e \neq e'
\end{cases}
\]
for any $\partial\in \mathrm{VF}^{(e-1,e'-1)}\setminus
\mathrm{CVF}^{(e-1,e'-1)}$. In particular the codimension of the
inclusion $\mathrm{Lie}_{\mathrm{alg}}^{\Z_d}(\C^2) \subseteq
\mathrm{VF}_{\mathrm{alg}}^{\Z_d}(\C^2)$ is $|J|$ if $e=e'$ and
$|J| -1$ otherwise.
\end{theorem}

Remark that $\dim_\CC \mathrm{CVF}^{(i,j)}=1$ and
$\dim_\CC\mathrm{VF}^{(i,j)}=2$ as a vector space. Hence, in the case
where $e\neq e'$ we have that $\mathrm{VF}^{(e-1,e'-1)}\subseteq
\mathrm{Lie}_{\mathrm{alg}}^{\Z_d}(\C^2)$. We postpone the proof of
this theorem to the end of this section.

The theorem immediately shows in which cases $\C^2$ has $\Z_d$-ADP or,
equivalently, $V_{d,e}$ has the strong ADP. It also allows in each
particular case to determine the values of $\ell$ from
Definition~\ref{ADP} for which $I^\ell\VFalg(X,\Xsing) \subseteq
\Liealg(X,\Xsing)$.

\begin{corollary} \label{Z-ADP} Let $V_{d,e}$ be a toric surface.

  \begin{enumerate}[$(i)$]
  \item $V_{d,e}$ has the strong ADP if and only if if and only if $e\
    \vert \ d+1$ and $e^2 \neq d + 1$.
  \item $V_{d,e}$ has the ADP and an upper bound for the minimal
    $\ell$ such that $I^\ell\VFalg(X,\Xsing) \subseteq
    \Liealg(X,\Xsing)$ is $e+e'-2$.
  \end{enumerate}

\end{corollary}

The next lemma shows what is happening if we take the Lie bracket of
two complete homogeneous vector fields.
\begin{lemma} \label{lembrackets} %
  Let $\partial_1\in \mathrm{CVF}^{(i,j)}$, $\partial_2 \in
  \mathrm{CVF}^{(i',j')}$, $\partial_3\in\mathrm{LND}_u^k$ and $\partial_4 \in
  \mathrm{LND}_v^{k'}$, then
  \begin{enumerate}[$(i)$]
  \item $[\partial_1,\partial_2] \in \mathrm{CVF}^{(i+i',j+j')}$,
  \item $[\partial_1,\partial_3] \in \mathrm{VF}^{(i-1,j+ke')}\setminus \mathrm{CVF}^{(i-1,j+ke')}$,
  \item $[\partial_1,\partial_4] \in \mathrm{VF}^{(i+k'e,j-1)}\setminus
    \mathrm{CVF}^{(i+k'e,j-1)}$,
  \item $[\partial_3,\partial_4] \in \mathrm{VF}^{(k'e-1,ke'-1)}$. Furthermore,
    $[\partial_3,\partial_4] \in \mathrm{CVF}^{(k'e-1,ke'-1)}$ if and only if
    $ek'= e'k$.
  \end{enumerate}
\end{lemma}
\begin{proof}
  All four statements follow by direct computation using
  Corollary~\ref{cor-gl-int} and Lemma~\ref{commutator}.
\end{proof}

The next two lemmas show
$\mathrm{Lie}(S)=\mathrm{Lie}_{\mathrm{alg}}^{\Z_d}(\C^2)$, each of
them showing one inclusion.
\begin{lemma}
 $S\subset \mathrm{Lie}_{\mathrm{alg}}^{\Z_d}(\C^2)$
\end{lemma}
\begin{proof}
  Take $(i,j)\in I \setminus J$, then either $(i-e,j+1) \in I$ or
  $(i+1,j-e')\in I$. In the first case pick $\partial\in
  \mathrm{CVF}^{(i-e,j+1)}$ and $\delta\in\mathrm{LND}_v^1$ and by Lemma
  \ref{lembrackets} we have $[\partial,\delta]\in\mathrm{VF}^{(i,j)}\setminus
  \mathrm{CVF}^{(i,j)}$ and thus $\mathrm{VF}^{(i,j)}\subset
  \mathrm{Lie}_{\mathrm{alg}}^{\Z_d}(\C^2)$. The second case works
  similarly.
\end{proof}
\begin{lemma}
  $\lbrace${invariant complete algebraic vector
    fields}$\rbrace \subset \mathrm{Lie}(S)$.
\end{lemma}
\begin{proof}
  Let $L$ be the set of vector fields appearing in the list of
  Theorem~\ref{thmlist}. We will first show that $L\subset S$. Let
  $\partial \in L$ and $\partial = \sum \partial_{i,j}$ its
  decomposition into homogeneous parts with respect to the standard
  grading on $\C^2$.  We directly see that all homogeneous parts of
  vector fields (1) and (2a) are complete.  For the vector
  fields (2b) and (3) we claim that $\partial_{i,j}=0$ whenever
  $(i,j)\in J$. Indeed, assume that $\partial_{i,j}\neq 0$ with $(i,j)
  \neq (0,0)$ and $\partial_{i,j}$ is not an LND. Then in case (2b) we
  have $e=e'=2d'+1$, $i+j\geq 4d'$ and $i\neq j$ since for every
  monomial $\mathfrak{m}$ of the polynomial $A$ we have
  $\deg_u\mathfrak{m}-\deg_v\mathfrak{m}$ is a multiple of 4. Hence,
  either $i>e$ or $j>e'$.  In case (3a) under the same assumptions we
  have $i> m+nl-l\geq m\geq d> e$. Similarly, in case (3b) we have $j>
  m+ln -l\geq m\geq d> e'$.

  \medskip

  In order to conclude the proof we only need to show that for a
  vector field $\delta \in \mathrm{Lie}(S)$ and an equivariant
  automorphism $\phi$ the vector field $\phi_*\delta \in
  \mathrm{Lie}(S)$.  By Lemma 4.10 in \cite{ArZa} $\phi$ is a
  composition of equivariant Jonqui\`eres automorphisms or more
  precisely it is a composition of linear equivariant automorphisms
  and flow maps of the vector fields $u^{ke}\partial/\partial v$ and $
  v^{ke'} \partial/\partial u$ (which are contained in $S$). First we
  show that for any linear automorphism $\phi$ we have $\phi_*\delta
  \in \mathrm{Lie}(S)$. For $e=1$ this statement is true for obvious
  reasons, indeed here we already have
  $\mathrm{Lie}(S)=\mathrm{Lie}_{\mathrm{alg}}^{\Z_d}(\C^2)$. For $e
  \neq 1$ all equivariant linear automorphisms are of the form
  $(u,v)\mapsto (au,bv)$ so they act by homothety on homogeneous
  vector fields of $\mathrm{Lie}(S)$. Now, if $\phi^t$ is the flow of
  the LND $\partial$ then $\phi^t_*\delta \in$ Lie($\partial,\delta$)
  for all $t$, since the Taylor expansion of $\phi^t_*\delta$ gives
  $\phi^t_*\delta = \delta + t [\partial, \delta] + \frac{1}{2} t^2
  [\partial, [\partial, \delta]] + \ldots + \frac{1}{n !} t^n
  [\partial, \ldots [\partial, \delta]] \ldots]$ which is a finite sum
  since $\partial$ is an LND and hence its flow is algebraic in
  $t$. Since $\partial\in S$ the claim follows.
\end{proof}

\begin{proof}[Proof of Theorem \ref{thminvliealg}]
  It is left to show that $\mathrm{Lie}(S)=S$ if $e=e'$ and
  $\mathrm{Lie}(S)=S\oplus \langle \partial \rangle$ if $e\neq e'$ for
  any $\partial\in \mathrm{VF}^{(e-1,e'-1)}\setminus
  \mathrm{CVF}^{(e-1,e'-1)}$. Let $(i,j)\in J$, then we need to show
  that $\mathrm{VF}^{(i,j)}\nsubseteq\mathrm{Lie}(S)$ unless $e\neq
  e'$ and $(i,j)=(e-1,e'-1)$. Assume
  $\mathrm{VF}^{(i,j)}\subset\mathrm{Lie}(S)$, then Lemma
  \ref{lembrackets} implies the existence of
  $\partial\in\mathrm{VF}^{(i,j)}\setminus \mathrm{CVF}^{(i,j)}$ such
  that $\partial=[\partial_1,\partial_2]$ for some
  $\partial_1\in\mathrm{LND}_u^1$ and $\partial_2\in\mathrm{LND}_v^1$,
  $e\neq e'$ and $(i,j)=(e-1,e'-1)$.
\end{proof}

\section{Implications of the algebraic density property for the
  holomorphic automorphism group}

We start with the obvious holomorphic version of Definition~\ref{ADP}.
Let $X$ be a Stein space and let $\Xsing$ be the singular locus. We
also let $Y\subseteq X$ be closed analytic subvariety of $X$
containing $\Xsing$ and let $\II=I(Y)\subseteq \OO(X)$ be the ideal of
$Y$.  Let $\VFhol(X,Y)$ be the $\OO(X)$-module of holomorphic vector
fields vanishing in $Y$ i.e., $\VFhol(X,Y)=\{\partial
\mid \partial(\OO(X))\subseteq \II\}$. Let $\Liehol(X,Y)$ be the Lie
algebra generated by all the complete vector fields in $\VFhol(X,Y)$.

\begin{definition} %
  We say that $X$ has the strong density property (DP) relative to $Y$
  if $\Liehol(X,Y)$ is dense in $\VFhol(X,Y)$ in the compact-open
  topology. Furthermore, we say that $X$ has the DP relative to $Y$ if
  there exists $\ell\geq 0$ such that $\II^\ell\VFhol(X,Y)$ is
  contained in the closure of $\Liehol(X,Y)$. With this definition,
  the DP relative to $Y$ with $\ell=0$ is just the strong DP relative
  to $Y$.
\end{definition}

\begin{proposition} \label{GAGA}
  Let $X$ be an affine algebraic variety and let $Y$ be a subvariety
  containing $\Xsing$. Then the ADP for $X$ relative to $Y$ implies
  the DP for $X$ relative to $Y$.
\end{proposition}

\begin{proof}
  The proposition follows from the fact that $I^\ell\VFalg(X,Y)$ is
  dense in $\II^\ell\VFhol(X,Y)$. Indeed, by Theorem A of Cartan, there
  are finitely many global sections $s_1,\ldots, s_N$ of
  $I^\ell\VFalg(X,Y)$ that generate the stalk at every point. A
  standard application of Theorem B of Cartan implies that any
  holomorphic section $s_h\in\II^\ell\VFhol(X,Y)$ over an
  $\OO(X)$-convex compact $K\subseteq X$ can be written as
  $s_h=f_1s_1+\ldots+f_Ns_N$ with $f_i\in \OO(K)$. By approximating
  the functions $f_i$ by global functions in $\CC[X]$, this implies
  $I^\ell\VFalg(X,Y)$ is dense in $\II^\ell\VFhol(X,Y)$.
\end{proof}

\begin{theorem}[\bf Relative Anders\'en-Lempert
  theorem] \label{AL-Theorem} %
  Let $X$ be a Stein space with the DP relative to a
  closed analytic subvariety $Y$ containing $\Xsing$. Let $\Omega$ be
  an open subset of $X$. Suppose that $ \Phi : [0,1] \times \Omega \to
  X$ is a $C^1$-smooth map such that

  \begin{enumerate} [(i)]
  \item $\Phi_t : \Omega \to X$ is holomorphic and injective for every
    $ t\in [0,1]$,

  \item $\Phi_0 : \Omega \to X$ is the natural embedding of $\Omega$
    into $X$, 

  \item $\Phi_t (\Omega)$ is a Runge subset of $X$ for every $t\in
    [0,1]$, and

  \item \label{fixing} $\Phi_t (\Omega)$ fixes $Y$ up to order $\ell$, where $\ell$
    is such that $\II^\ell\VFhol(X,Y)$ is contained the closure of
    $\Liehol(X,Y)$.

    Then for each $\epsilon >0 $ and every compact subset $K \subset
    \Omega$ there is a continuous family, $\alpha: [0, 1] \to
    \Aut_{hol} (X)$ of holomorphic automorphisms of $X$ fixing $Y$
    pointwise such that $$\alpha_0 = id \, \, \, {\rm and} \, \,
    \,\vert \alpha_t - \Phi_t \vert_K <\epsilon {\rm \ \ for\ \ every\
      \ } t \in [0,1]$$ 
  \end{enumerate}
\end{theorem}

Point $(iv)$ in the assumptions of the theorem means the following:
Consider the time dependent vector field
$V(x,t_0)=\left.\frac{d}{dt}\right|_{t=t_0}\Phi_t(\Phi_{t_0}^{-1}(x))$. The
isotopy $\Phi_t (\Omega)$ fixes $Y$ up to order $\ell$ if $V(x,t_0)$
is a section of $\II^\ell\VFhol(X,Y)$ over $\Phi_{t_0}(\Omega)$ for
all $t_0$.

\begin{proof}[Sketch of proof]
  The map $\Phi_{t_0}$ is the $t_0$-map of the time dependent vector
  field $V(x,t)$. It can be approximated by dividing the time interval
  into small pieces and integrating the time independent vector fields
  over each piece. By assumption, each of those time independent
  fields is a section in
  $\II^\ell\VFhol(X,Y)(\Phi_{t_0}(\Omega))$. Since the sheaf
  $\II^\ell\VFhol(X,Y)$ is coherent, a similar use of Theorem A and B
  of Cartan as in the proof of Proposition~\ref{GAGA} leads to the
  fact that these time independent vector fields in the Runge domain
  $\Phi_{t_0}(\Omega)$ can be approximated by global vector fields in
  $\II^\ell\VFhol(X,Y)$. By assumption, these vector fields can be
  approximated by Lie combinations of complete vector fields vanishing
  in $Y$ (not necessarily in $\II^\ell\VFhol(X,Y)$). Now the standard
  use of Euler's method gives the desired conclusion.
\end{proof}

\begin{remark}
  If $Y\cap \Phi_t(\Omega)=\emptyset$ for all $t\in [0,1]$, then
  condition~\eqref{fixing} in Theorem~\ref{AL-Theorem} is trivially
  satisfied.
\end{remark}

\begin{corollary}
  Any smooth point in an affine toric variety $X$ of dimension $n\geq
  2$ different from the torus has an open neighborhood in the
  Euclidean topology biholomorphic to $\CC^n$.
\end{corollary}

\begin{proof}
  Let $x\in X$. Take a Runge neighborhood $U$ of $x$ biholomorphic to
  the unit ball sending $x$ to zero and let $\Phi_t$ be the map
  $\left(1-\frac{t}{2}\right)z$ in the unit ball. Since $X$ has the DP
  relative to $\Xsing$, Theorem~\ref{AL-Theorem} implies that these
  contractions can be approximated by holomorphic automorphisms
  $\alpha_t$ of $X$ (fixing $\Xsing$ pointwise). The automorphism
  $\alpha_1$ has an attractive fixed point near $x$. The bassin of
  attraction of this point is biholomorphic to $\CC^n$
  \cite{RoRu88}. Since the holomorphic automorphism group of $X$ is
  transitive on $X\setminus \Xsing$, the claim follows.
\end{proof}

\bibliographystyle{alpha} \bibliography{density}

\end{document}